\documentclass[11pt,reqno]{amsart}

\usepackage[margin=1.1in]{geometry}
\usepackage{amsmath,amssymb,amsthm}
\usepackage{booktabs}
\usepackage[hidelinks]{hyperref}

\theoremstyle{plain}
\newtheorem{theorem}{Theorem}
\newtheorem{lemma}[theorem]{Lemma}
\newtheorem{corollary}[theorem]{Corollary}
\theoremstyle{remark}
\newtheorem{remark}[theorem]{Remark}

\newcommand{\zz}{z}

\begin{document}

\title[Five improved lower bounds for $\zz(m,n;3,3)$]
      {Five improved lower bounds for Zarankiewicz numbers $\zz(m,n;3,3)$}

\author{Abhishek Saurabh}
\address{Independent researcher}
\email{saurabh.abhishek1985@gmail.com}

\date{\today}


\begin{abstract}
We record five improved lower bounds for Zarankiewicz numbers with $s=t=3$:
\[
  \zz(13,19;3,3)\ge 118,\quad
  \zz(14,19;3,3)\ge 126,\quad
  \zz(16,18;3,3)\ge 136,
\]
\[
  \zz(14,20;3,3)\ge 126,\quad
  \zz(16,19;3,3)\ge 136 .
\]
The first three are certified by explicit $K_{3,3}$-free $0/1$ matrices; the last two follow from the
second and third by monotone padding. Compared with the lower bounds compiled in Figure~2 of
\cite{BNL}, namely $114$, $121$, $130$, $125$ and $132$, the improvements are $+4$, $+5$, $+6$, $+1$
and $+4$ respectively. The three matrices were produced by a kicked-greedy local search operated
autonomously by a discovery system and were verified exactly, by exhaustive inspection of every
$3\times3$ row/column triple, in several mutually independent implementations; they were also
re-verified independently by the authors of \cite{BNL} using their own verifier in July 2026. All three
witnesses are printed in full in Appendix~\ref{sec:witnesses} and accompany this note as
machine-readable ancillary files together with standalone, dependency-free verifiers.
\end{abstract}

\maketitle

\section{Introduction}

For positive integers $m,n,s,t$, the \emph{Zarankiewicz number} $\zz(m,n;s,t)$ is the maximum number
of $1$s in an $m\times n$ matrix with entries in $\{0,1\}$ that contains no all-ones $s\times t$
submatrix, i.e.\ no $s$ rows and $t$ columns whose $st$ common entries are all $1$. Equivalently, it
is the maximum number of edges in a bipartite graph with parts of sizes $m$ and $n$ containing no
$K_{s,t}$ with the $s$-side inside the part of size $m$. Throughout this note $s=t=3$, so the
convention is symmetric and no ambiguity arises; we abbreviate $\zz(m,n)=\zz(m,n;3,3)$ where
convenient. Determining $\zz(m,n;s,t)$ is the Zarankiewicz problem, posed by
Zarankiewicz~\cite{Zar51}; the classical general upper bound is due to K\H{o}v\'ari, S\'os and
Tur\'an~\cite{KST54}.

For $s=t=3$ the exact values are known only in a limited range, and the current state of knowledge is
a table of lower and upper bounds. The lower bounds are, with few exceptions, produced by computer
search over $0/1$ matrices; the upper bounds come from counting arguments, integer/linear programming
relaxations, and exhaustive or SAT-based reasoning. The bounds relevant here are drawn from the
following sources.

\begin{itemize}
\item Collins, Riasanovsky, Wallace and Radziszowski~\cite{CRWR16} give, in their Table~4, upper bounds
      that remain the tightest published values in several cells of interest. In particular
      $\zz(16,18;3,3)\le 140$, which is tighter than the value $146$ appearing for that cell in the
      compilation of Figure~2 of~\cite{BNL}.
\item Tan~\cite{Tan22} applies SAT solving to
      Zarankiewicz numbers.
\item Davies, Gill and Horsley~\cite{Dav24} obtain
      upper bounds by linear programming.
\item Bhan, Nobili and Langer~\cite{BNL} obtain lower bounds by an
      LLM-driven evolutionary search. Their Figure~2 is a per-cell grid giving, for each $(m,n)$ in
      the range considered, a known upper bound together with the lower bound achieved by their own
      construction. These are the lower bounds that the present note improves, and all
      ``previous best'' values quoted below refer to that figure.
\end{itemize}

Two further papers appeared in August 2026, after the constructions reported here were found and
communicated: Hou~\cite{Hou26} and
Afrasyab~\cite{Afr26} close or improve a number of
cells. Neither paper addresses any of the five cells treated here: Hou's results concern
$(12,18)$, $(13,17)$, $(13,18)$, $(14,17)$, $(14,18)$, $(15,17)$ and $(15,18)$, and Afrasyab's
concern $(12,n)$ for $18\le n\le 22$ together with $(13,22)$, $(13,18)$, $(14,17)$, $(14,18)$,
$(15,17)$, $(15,18)$ and $(16,17)$. We comment on the overlap at $(13,17)$ in
Remark~\ref{rem:1317}; none of the five bounds claimed here is affected by those papers.

The contribution of this note is deliberately narrow: five improved lower bounds, each backed by an
explicit witness that a reader can verify in under a second on a laptop. We make no claim of novelty
for the search method, which is a standard kicked-greedy local search, and no claim about upper
bounds.

\section{The bounds}

\begin{theorem}\label{thm:main}
There exist $K_{3,3}$-free $0/1$ matrices realising
\[
  \zz(13,19;3,3)\ \ge\ 118,\qquad
  \zz(14,19;3,3)\ \ge\ 126,\qquad
  \zz(16,18;3,3)\ \ge\ 136 .
\]
\end{theorem}

\begin{proof}
The three matrices $A_{13\times19}$, $A_{14\times19}$ and $A_{16\times18}$ are displayed in
Appendix~\ref{sec:witnesses}. Each is checked to have entries in $\{0,1\}$, to have the stated number
of $1$s, and to contain no all-ones $3\times 3$ submatrix; the last is decided by inspecting every
one of the $\binom{m}{3}\binom{n}{3}$ triples of rows and triples of columns, which is $277{,}134$,
$352{,}716$ and $456{,}960$ checks respectively. Appendix~\ref{sec:verify} gives the recipe.
\end{proof}

The remaining two bounds are corollaries of Theorem~\ref{thm:main} and the following elementary
monotonicity, which we state and prove only because we rely on it explicitly.

\begin{lemma}[Monotone padding]\label{lem:pad}
For all positive integers $m,n,s,t$,
\[
  \zz(m,n;s,t)\ \le\ \zz(m,n+1;s,t)
  \qquad\text{and}\qquad
  \zz(m,n;s,t)\ \le\ \zz(m+1,n;s,t).
\]
\end{lemma}

\begin{proof}
Let $A$ be an $m\times n$ $0/1$ matrix with no all-ones $s\times t$ submatrix and with
$\zz(m,n;s,t)$ ones. Form $A'$ by appending one all-zero column, so $A'$ is $m\times(n+1)$ and has
the same number of $1$s. Let $S$ be a set of $s$ rows and $T$ a set of $t$ columns of $A'$. If $T$
avoids the new column then the corresponding submatrix of $A'$ is a submatrix of $A$ and hence is not
all-ones; if $T$ contains the new column then that column contributes an entry $0$, so the submatrix
is again not all-ones. Thus $A'$ is a valid $m\times(n+1)$ configuration, giving
$\zz(m,n+1;s,t)\ge\zz(m,n;s,t)$. Appending an all-zero row gives the second inequality.
\end{proof}

\begin{corollary}\label{cor:pad}
$\zz(14,20;3,3)\ \ge\ \zz(14,19;3,3)\ \ge\ 126$ and
$\zz(16,19;3,3)\ \ge\ \zz(16,18;3,3)\ \ge\ 136$.
\end{corollary}

\begin{proof}
Immediate from Theorem~\ref{thm:main} and Lemma~\ref{lem:pad}. Explicitly, the witnesses are
$A_{14\times19}$ with one all-zero column appended, and $A_{16\times18}$ with one all-zero column
appended; both padded matrices are supplied as ancillary files with their own verifiers.
\end{proof}

We stress that Corollary~\ref{cor:pad} is bookkeeping, not search: it costs no computation, and it
improves the compiled lower bounds for those two cells only because the compilation predates
Theorem~\ref{thm:main}, so the padding implied by the new bounds had not yet been applied to them.

\section{Method}

The three matrices of Theorem~\ref{thm:main} were found by a kicked-greedy local search over $0/1$
matrices of fixed shape: a greedy fill that adds a $1$ only where it creates no all-ones $3\times 3$
submatrix, run to saturation, followed by a randomised ``kick'' that deletes a small random set of
$1$s and re-saturates, with the incumbent retained on non-improvement. The search is classical: no
language model, no learned component, and no problem-specific algebraic input. The objective is the
exact count of $1$s, and $K_{3,3}$-freeness is enforced as a hard constraint at every step rather
than penalised, so every intermediate state --- and in particular every state that is ever recorded
--- is a valid witness.

The search ran inside an autonomous research system (NOVA) that selected the target cells, ran the
searches, and packaged the results; the mathematical content of the search itself is standard. The
system's cell-selection rule is the reason two of the three cells were attacked at all: it targets
cells whose compiled lower bound is dominated by the monotone floor implied by a smaller cell
(Lemma~\ref{lem:pad}), on the grounds that such cells are the ones where the incumbent construction
has most visibly not converged.

Provenance of the individual runs, as recorded at the time:
\begin{itemize}
\item $A_{13\times19}$: found 2026-07-05, seed $2$, $215{,}624$ kicks, wall-clock budget
      $600$ seconds for the cell.
\item $A_{16\times18}$: found 2026-07-05, seed $0$, $130{,}287$ kicks, same budget.
\item $A_{14\times19}$: found during a campaign of 2026-07-03/04, seed $7$; that build of the system
      did not record a kick count.
\end{itemize}
We note that a wall-clock-budgeted stochastic search is not bit-reproducible from a seed
alone, so we do not claim that re-running the search reproduces these matrices. Nothing in this note
depends on that: the claims are statements about three explicit matrices, and those matrices are
verifiable exactly and cheaply.

\medskip
\noindent\textbf{Verification.} Each matrix was verified by exhaustive enumeration of all
$\binom{m}{3}\binom{n}{3}$ row-triple/column-triple pairs, with the count of $1$s recomputed from the
matrix rather than read from any label. Verification was performed by several implementations that
share no code: an array-based checker, a column-pair/column-triple set-arithmetic checker, and a
plain triple-nested loop in the Python standard library, all agreeing. In July 2026 the three
matrices, together with a fourth ($13\times17$; see Remark~\ref{rem:1317}), were communicated to the
authors of~\cite{BNL}, who ran their own verifier on all four and reported no errors, and who
subsequently confirmed in correspondence that the bounds improve on those compiled in their
Figure~2.

\section{Summary of the five cells}

Table~\ref{tab:results} collects the five cells. ``Previous LB'' is the lower bound compiled in
Figure~2 of~\cite{BNL}. ``Best known UB'' is the smallest published upper bound we are aware of, with
its source given per cell: for four of the five cells this is the value appearing in the compilation
of Figure~2 of~\cite{BNL}; for $(16,18)$ it is the tighter value from Table~4 of~\cite{CRWR16}.
``Gap'' is the difference between that upper bound and our lower bound.

\begin{table}[htbp]
\centering
\caption{The five cells. All values for $s=t=3$.}
\label{tab:results}
\begin{tabular}{lccccl}
\toprule
Cell $(m,n)$ & Previous LB & New LB & Best known UB & Gap & Source of the UB \\
\midrule
$(13,19)$ & $114$ & $\mathbf{118}$ & $125$ & $7$  & Fig.~2 of~\cite{BNL} \\
$(14,19)$ & $121$ & $\mathbf{126}$ & $135$ & $9$  & Fig.~2 of~\cite{BNL} \\
$(16,18)$ & $130$ & $\mathbf{136}$ & $140$ & $4$  & Table~4 of~\cite{CRWR16} \\
$(14,20)$ & $125$ & $\mathbf{126}$ & $140$ & $14$ & Fig.~2 of~\cite{BNL} \\
$(16,19)$ & $132$ & $\mathbf{136}$ & $152$ & $16$ & Fig.~2 of~\cite{BNL} \\
\bottomrule
\end{tabular}
\end{table}

\noindent
The improvements over the previous lower bounds are therefore $+4$, $+5$, $+6$, $+1$ and $+4$.
The bounds for $(14,20)$ and $(16,19)$ are those of Corollary~\ref{cor:pad}; the other three are
those of Theorem~\ref{thm:main}.

The $(16,18)$ witness has a notably regular degree distribution: its row sums are eight $9$s and
eight $8$s, and its column sums are $(9^2,8^8,7^6,6^2)$, where exponents denote multiplicities. For
completeness, the corresponding profiles of the other two witnesses are: $A_{13\times19}$, rows
$(10^6,9^2,8^5)$ and columns $(8^1,7^8,6^6,5^2,4^2)$; $A_{14\times19}$, rows
$(11^1,10^3,9^5,8^5)$ and columns $(8^4,7^9,6^2,5^3,4^1)$.

\begin{remark}\label{rem:1317}
A fourth witness produced by the same program in July 2026 gives $\zz(13,17;3,3)\ge 110$, matching
the upper bound $\zz(13,17;3,3)\le 110$ of Table~4 of~\cite{CRWR16} and hence the exact value. That
matrix was found on 2026-07-03/04 and was communicated to, and verified by, the authors of~\cite{BNL}
in July 2026. The value $\zz(13,17;3,3)=110$ subsequently appears in~\cite{Hou26}. We record the
chronology only for context and claim no priority: \cite{Hou26} is the published source for that
value, and we do not restate it as a result of this note. It is mentioned because the same program
and the same search produced the three witnesses of Theorem~\ref{thm:main}.
\end{remark}

\subsection*{Data availability}
The three witnesses of Theorem~\ref{thm:main} and the two padded witnesses of
Corollary~\ref{cor:pad} are printed in Appendix~\ref{sec:witnesses} and are also supplied as
ancillary files with this submission, in two machine-readable forms (bitstrings and JSON) together
with one standalone verifier per witness. The verifiers are written against the Python standard
library only: no third-party packages, no network access, no build step. A file of SHA-256 hashes
covers every ancillary file.

\appendix

\section{The witness matrices}\label{sec:witnesses}

Each matrix is printed one row per line, as a string of $n$ characters in $\{\texttt{0},\texttt{1}\}$
read left to right. These strings are the primary form of the data; the ancillary files contain all
five matrices, these three and the two padded ones, in the identical bitstring form and again in
JSON.

\subsection*{\texorpdfstring{$A_{13\times19}$: $13\times19$, $118$ ones, $K_{3,3}$-free}
                            {A(13x19): 13 x 19, 118 ones, K(3,3)-free}}
\begin{verbatim}
1010010100111001000
0000100110101101100
0101110001110011100
0010100001111000011
0110111100100100010
1001101000011101010
1000110101000001011
1111000011100101001
1111100100001010101
1010110010010100100
0010001111010011110
1000011010101010111
0100011111011100001
\end{verbatim}

\subsection*{\texorpdfstring{$A_{14\times19}$: $14\times19$, $126$ ones, $K_{3,3}$-free}
                            {A(14x19): 14 x 19, 126 ones, K(3,3)-free}}
\begin{verbatim}
1110110100110010000
0010110101001001110
0001010110110000111
1000100011011110011
0100000000111011110
1101010110001101000
1111001001011000101
0110011010101000010
0001111000101110100
0010011010010101100
0101111000010001010
0011100010111001000
1011001100100111011
0100101111000011101
\end{verbatim}

\subsection*{\texorpdfstring{$A_{16\times18}$: $16\times18$, $136$ ones, $K_{3,3}$-free}
                            {A(16x18): 16 x 18, 136 ones, K(3,3)-free}}
\begin{verbatim}
100011101010010001
010000111100011101
110001011101100000
100000010111001011
101100000001111101
100110111000101010
110100001011010110
010010010011101100
110011000100100111
001001011001000111
011101010010010001
001111001110001100
011011100001011010
011000001010101011
000110100111100001
001000110110110110
\end{verbatim}

\subsection*{The padded witnesses}
The $14\times20$ witness for $\zz(14,20;3,3)\ge126$ is $A_{14\times19}$ with the character
\texttt{0} appended to every row; the $16\times19$ witness for $\zz(16,19;3,3)\ge136$ is
$A_{16\times18}$ with the character \texttt{0} appended to every row. Both are supplied explicitly
as ancillary files so that no reader has to perform the transformation by hand.

\section{Verification recipe}\label{sec:verify}

Every witness ships with its own self-contained verifier script, which embeds the matrix and
re-derives everything from it. With any Python~3 interpreter, and with no third-party packages and
no network access:
\begin{verbatim}
    python verify_13x19.py
\end{verbatim}
and likewise for the other four. Each script (i) checks the shape and that every entry lies in
$\{0,1\}$; (ii) recomputes the number of $1$s from the matrix rather than reading it from a label,
failing if the recount disagrees; (iii) enumerates all $\binom{m}{3}\binom{n}{3}$ pairs consisting of
three rows and three columns and fails if any such $3\times3$ submatrix is all-ones; and
(iv) compares the recomputed count with a recorded threshold. Each exits with status $0$ if and only
if all four checks pass, and prints the reason for any failure. The two verifiers for the padded
witnesses additionally delete the appended all-zero column and re-verify the underlying witness, so
that Lemma~\ref{lem:pad} is checked rather than assumed.

The three verifiers for the searched witnesses are the scripts written at the time of discovery and
are shipped unaltered. They therefore compare against an internal threshold of $115$, $124$ and
$132$ rather than against the published lower bounds $114$, $121$ and $130$, and they describe the
results in the provisional language used before the constructions had been checked by anyone else.
That threshold is the larger of the published lower bound and the bound already implied by a smaller
cell through Lemma~\ref{lem:pad}, so it is the stricter of the two and passing it implies passing the
published bound. Neither the threshold nor the wording affects what the scripts verify.

A reader who prefers not to run supplied code can transcribe the strings of
Appendix~\ref{sec:witnesses} and apply the same four checks in any language; the whole computation is
a few hundred thousand elementary tests per matrix.


\end{document}